\documentclass[12pt]{amsart}
\usepackage{amsfonts}
\usepackage{amssymb}
\usepackage{mathtools}
\usepackage[all]{xy}
\usepackage{amsmath, amssymb, amsfonts, amstext, amsthm}
\usepackage[mathscr]{euscript}
\usepackage{enumerate}
\usepackage{verbatim}
\def \mcO{\mathcal{O}}
\def \mcF{\mathcal{F}_{D,A,V}}
\def \mcE{\mathcal{E}}

\def \mcQ{Q}
\def \tildeQ{\widetilde{Q}}
\def \mbP{\mathbb{P}}
\def \mcI{\mathcal{I}_{Z(V)}}

\def \mcU{\mathcal{U}}

\def \mcL{L}
\def \tilde*F{\widetilde{\mathcal{F}}}
\def \tildeU{\widetilde{U}}

\newcommand{\xra}{\xrightarrow}

\newcommand{\trm}{\textrm} 
\newcommand{\lr}{\longrightarrow}
\newcommand{\hra}{\hookrightarrow}
\newcommand{\mbb}{\mathbb}

\def \mcID{\mathcal{I}_{Z_D(V)}}

\newtheorem{stabonprojintro}{Theorem}[section]
\newtheorem{everycaseintro}[stabonprojintro]{Theorem}
\newtheorem{calabiyauintro}[stabonprojintro]{Theorem}
\newtheorem{Lefschetz}{Remark}[section]
\newtheorem{prelim}[Lefschetz]{Proposition}

\newtheorem{stabonproj}{Remark}[section]
\newtheorem{flatHtoD}[stabonproj]{Lemma}
\newtheorem{pbfromhpisstable}[stabonproj]{Lemma}

\newtheorem{reflexstab}[stabonproj]{Lemma}

\newtheorem{alterpf}[stabonproj]{Corollary}
\newtheorem{exceptions}[stabonproj]{Remark}
\newtheorem{notsemist}[stabonproj]{Lemma}

\newtheorem{FonD}[stabonproj]{Remark}
\newtheorem{hpstable}{Remark}[section]
\newtheorem{Kernelbundle}[hpstable]{Proposition}

\title{Lazarsfeld-Mukai Reflexive Sheaves and their Stability}
\author[P. Narayanan]{Poornapushkala Narayanan}
\address{Department of Mathematics, Indian Institute of Technology Madras, Chennai - 600036.}
\email{poorna.p.narayanan@gmail.com}
\thanks{Mathematics Classification numbers: 14C20, 14J60, 14M10}

\begin{document}

\begin{abstract}
  Consider an ample and globally generated line bundle $L$ on a smooth
  projective variety $X$ of dimension $N\geq 2$ over $\mathbb{C}$. Let
  $D$ be a smooth divisor in the complete linear system of $\mcL$. We
  construct reflexive sheaves on $X$ by an elementary transformation
  of a trivial bundle on $X$ along certain globally generated
  torsion-free sheaves on $D$. The dual reflexive sheaves are called
  the \emph{Lazarsfeld-Mukai} reflexive sheaves. We prove the
  $\mu_{\mcL}$-(semi)stability of such reflexive sheaves under certain
  conditions.
\end{abstract}

 \maketitle
 \section{Introduction}
 \emph{Lazarsfeld-Mukai} bundles were introduced by Lazarsfeld
 \cite{RL} and Mukai \cite{Mu} in the 1980s. They are an important
 class of vector bundles obtained from certain elementary
 transformations and have found applications in studying syzygies and
 Brill-Noether theory. These bundles play a crucial role in
 Lazarsfeld's proof of the Gieseker-Petri theorem \cite{RL} and
 Voisin's proof of the generic Green's conjecture \cite{CV1,CV2}.
 
 Suppose $X$ is a smooth projective surface over $\mathbb{C}$, and $C$
 is a smooth, irreducible curve on $X$. Consider a globally generated
 line bundle $A$ on $C$. Denote by $i_*A$, the direct image of $A$ on
 $X$ where $i:C\hra X$ is the inclusion. Then $i_*A$ is a globally
 generated coherent sheaf on $X$. We thus have the following exact
 sequence on $X$ where the kernel $F$ is a vector bundle:
 $$0\lr F\lr H^0(A)\otimes\mcO_X\xra{ev} i_*A\lr 0\,.$$
 The dual of $F$ is called the Lazarsfeld-Mukai bundle on $X$
 associated to the pair $(C,A)$. Lelli-Chiesa \cite{ML} has studied
 the (semi)stability of the Lazarsfeld-Mukai bundles on K3-surfaces,
 and similar results have been obtained by us on abelian surfaces
 \cite{NP}. Also, \cite{AP} and references therein give a general
 survey of Lazarsfeld-Mukai bundles with other applications.

 In this article we generalize the above construction to higher
 dimensional varieties. We in fact obtain reflexive sheaves as
 kernels. We study their $\mu$-(semi)stability properties in various
 cases. This construction also enables us to obtain on any smooth
 projective variety $X$, semistable vector bundles $E$ with
 $\trm{rank}\,E=\trm{dim}\,X$.

 Suppose $X$ is a smooth projective variety of dimension $N\geq 2$
 over $\mathbb{C}$. Let $D\xhookrightarrow{i} X$ be a smooth,
 irreducible divisor on $X$ and $A$ be an ample and globally generated
 line bundle on $D$. Consider a general subspace $V\subset H^0(D,A)$
 of dimension $r\geq 2$.  Let $Z(V)\hra D$ be the closed subscheme
 defined by the vanishing of sections of $V$ and $\mcI\subset\mcO_D$
 be its ideal sheaf. Then $A\otimes\mcI$ is a globally generated
 torsion-free sheaf on $D$. We have the following short exact
 sequence, which defines the sheaf $\mcF$ associated to the triple
 $(D,A,V)$ on $ X$:
 \[0\lr\mcF\lr V\otimes\mcO_X\lr i_*(A\otimes\mcI)\lr 0\,.\] The
 kernel $\mcF$ is a reflexive sheaf of rank $r$ on $X$, whose dual is
 called the \emph{Lazarsfeld-Mukai} reflexive sheaf (see
 $\mathcal{x}\,$\ref{construct}).

 We remark that the same construction can be carried out under the
 weaker assumption that $D$ is just reduced and irreducible but not
 necessarily smooth. But for the purpose of this paper, we confine
 ourselves mainly with smooth and irreducible divisors $D$.

 The $\mu$-(semi)stability properties of the sheaves $\mcF$ are
 studied in $\mathcal{x}\,$\ref{studystab}. The first case we consider
 is that of a variety $X$ whose Picard group is cyclic.
\begin{stabonprojintro}\label{stabonprojintro}
  Suppose $X$ is an irreducible smooth projective variety over
  $\mbb{C}$ of dimension $N\geq 2$, such that
  $\emph{Pic}\,X=\mathbb{Z}\cdot [H]$, where $[H]$ is the class of an
  ample divisor. Let $D\in|\mcO_X(H)|$ and $A$ be an ample, globally
  generated line bundle on $D$. Consider $V\subset H^0(D,A)$, an
  $r$-dimensional subspace where $r\geq 2$. Then the reflexive sheaf
  $\mcF$ is $\mu_{H}$-stable.
\end{stabonprojintro}
Any $D\in|\mcO_X(H)|$ is reduced, irreducible and
Cohen-Macaulay. Hence we can consider reflexive sheaves $\mcF$ for all
such $D$. In the specific case when $X$ is the projective space, we
have the following theorem.
\begin{everycaseintro}\label{finalprojintro}
  Suppose $X=\mbb{P}^N_{\mbb{C}}$ for $N\geq 2$. Consider
  $\mcL=\mcO(d)$ with $d>0$ on $X$.  Let $D\in |\mcL|$ be a
  \emph{general} smooth, irreducible hypersurface and $A$ be an ample,
  globally generated line bundle on $D$.  Suppose that $V$ is a
  \emph{general} $r$-dimensional subspace of $H^0(D,A)$, where
  $r\geq 2$.  Then the following table summarizes the conditions on
  $L$, $A$ and $r$ under which the sheaves $\mcF$ are
  $\mu_{\mcO(1)}$-(semi)stable.
\begin{center}
\begin{tabular}{|c|c|c|c|c|}
\hline
 & $L$ & which $A$ & $r$ &  {stability}  \\
\hline
(a) & $L=\mcO(1)$ & all $A$ & $r\geq 2$ & $\mu_{\mcO(1)}$-stable \\
(b) & $L=\mcO(2)$ & all $A$ & $r=2$ & $\mu_{\mcO(1)}$-semistable \\
(c) & $L=\mcO(2l)$ & $A=\mcO(l)|_D$ & $r=2$ & $\mu_{\mcO(1)}$-semistable \\
(d) & $L=\mcO(d)$ & $A=\mcO(md)|_D$  & $2\leq r\leq {N-1+m\choose m}$ & $\mu_{\mcO(1)}$-semistable \\
  & when $d>1$    &        &                                         &  \\
\hline
\end{tabular}
\end{center}
\end{everycaseintro}
See $\mathcal{x}\,$\ref{caseprojspace} for a proof. Part (a) of the
above theorem follows from Theorem \ref{stabonprojintro}. Part (b) is
proved by applying a lemma from \cite{OK}. In case of part (c), we
prove that $\mcF|_D$ is $\mu_{\mcO(1)|_D}$-semistable, which implies
our result. We prove part (d) of the theorem by proving it in general
for any smooth, irreducible projective variety $X$, cf. Theorem
\ref{calabiyauintro}. We remark that, in case (d) of the above
theorem, if the condition $A=\mcO(md)|_D$ is weakened, the assertion
is not necessarily true. Lemma \ref{notsemist} in
$\mathcal{x}\,$\ref{caseprojspace} gives a class of such
examples. Note that parts (a) and (b) of the theorem hold for all
reduced and irreducible $D$ in the linear system and all
$r$-dimensional subspaces $V$.

We prove the following theorem on the $\mu$-(semi)stability of the
reflexive sheaves $\mcF$ on arbitrary smooth projective varieties.
\begin{calabiyauintro}\label{calabiyauintro}
  Suppose $X$ is an irreducible, smooth projective variety of
  dimension $N\geq 2$ over $\mbb{C}$. Consider an ample, globally
  generated line bundle $L$ on $X$ and an irreducible, smooth
  $D\in |\mcL|$. For $m>0$, let $V\subset H^0(D,\mcL|_D^{\otimes m})$
  be an $r$-dimensional subspace, where
  $2\leq r\leq {N-1+m\choose m}$. Then, for a general pair $(D,V)$,
  the reflexive sheaf $\mathcal{F}_{D,\mcL|_D^{\otimes m},V}$ is
  $\mu_{\mcL}$-semistable.
\end{calabiyauintro}
See $\mathcal{x}\,$\ref{irred}. The method of proof employed is the
following. Consider an appropriate finite morphism
$X\rightarrow\mbP^N$, where $N=\trm{dim}\,X$. We prove the
corresponding (semi)stability statement for the projective space. We
know by \cite[Lemma 1.17]{Mar}, that the pullback of a semistable
torsion-free sheaf under a finite morphism is semistable; and that
semistability is an open condition in flat families \cite[Proposition
2.3.1]{HL}. We thus get the required result.

The same technique is applied to study the (semi)stability of some
kernel bundles. A kernel bundle $M_{\mcL,W}$ is defined as
follows. Consider an ample and globally generated line bundle $\mcL$
on a smooth, irreducible projective variety $X$ of dimension $n$. Let
$W\subset H^0(X,\mcL)$ be a subspace such that the linear system
$\mathbb{P}W$ is base-point free. Hence we have the following short
exact sequence where $M_{\mcL,W}$ is the kernel vector bundle:
$$0\lr M_{\mcL,W}\lr W\otimes\mcO_X\lr \mcL\lr 0\,.$$
Let $W\subset H^0(X,L)$ be a general $(n+1)$-dimensional subspace. We
prove that the kernel bundle $M_{\mcL,W}$ associated to $(\mcL,W)$ is
$\mu_{\mcL}$-polystable, cf. Proposition \ref{kernel}.  We mention
that for curves, certain surfaces and projective spaces, the
$\mu_{\mcL}$-(semi)stability of $M_{\mcL,W}$ has been obtained in
\cite{AB,BUT,CC,CC1,LE,LEM,HF,EM,RP}, for $W=H^0(X,\mcL)$ with certain
assumptions on $L$.
\subsection*{Acknowledgements}
I thank Dr. Jaya NN Iyer for her guidance during the course of this
project. I also thank Dr. T. E. Venkata Balaji and Prof. D. S. Nagaraj
for helpful discussions and their support. I also thank the referee
and the editor for useful comments which helped make the exposition
better.
\section{Preliminaries}\label{Preliminaries}
\subsection{Definitions and Notations} Let $X$ be a smooth projective
variety of dimension $N\geq 2$ over $\mathbb{C}$.
\begin{enumerate}
\item Let $\mcL$ be an ample and globally generated line bundle on
  $X$. By Bertini's theorem, the set $\trm{sm}|L|$ as given below is a
  dense open set of $|L|$,
  \begin{equation*}\label{sm|L|}
  \trm{sm}|\mcL|=\{D\in |\mcL|:D\trm{ is smooth and irreducible}\}\,.
\end{equation*}
\item Let $W$ be a vector space over $\mathbb{C}$. Then $G(m,W)$
  ($1\leq m\leq \trm{dim}\,W$) denotes the Grassmannian of
  $m$-dimensional subspaces of $W$.
\end{enumerate}
\subsection{Mumford-Takemoto (semi)stability} Let $L$ be an ample line
bundle on $X$ (as above). Consider a torsion-free coherent sheaf $F$
of rank $r$ on $X$.
\begin{enumerate}
\item The slope of $F$ with respect to $L$ is defined as:
  \[\mu_L(F)=\frac{c_1(F)\cdot (L^{N-1})}{r}\,.\]
\item The sheaf $F$ is said to be $\mu_L$-semistable
  (resp. $\mu_L$-stable), if for any coherent subsheaf $E\subset F$ of
  rank $s$ where $0<s<r$, one has $\mu_L(E)\leq\mu_L(F)$
  (resp. $\mu_L(E)<\mu_L(F)$).
\item The torsion-free coherent sheaf $F$ is $\mu_L$-polystable if it
  is a direct sum of $\mu_L$-stable sheaves of the same slope.
\end{enumerate}
 \section{Construction of Reflexive sheaves}\label{construct}
 Consider a smooth projective variety $X$ of dimension $N\geq 2$ over
 $\mathbb{C}$, and an ample, globally generated line bundle $L$ on
 $X$. For a divisor $D\in\trm{sm}|\mcL|$, let $i:D\hra X$ denote the
 inclusion. Let $A$ be an ample, globally generated line bundle on
 $D$.
 
 By the Noether-Lefschetz Theorem, if $N\geq 4$, then
 $\text{Pic}\,X\lr\text{Pic}\,D$ is an isomorphism, cf. \cite[Example
 3.1.25]{Laz}. Hence, the line bundle $A$ is the restriction of a line
 bundle from $X$.

 Consider $G(r,H^0(D,A))$, the Grassmannian of $r$-dimensional
 subspaces of the space of global sections $H^0(D,A)$, where
 $2\leq r\leq h^0(D,A)$. For $V\in G(r,H^0(D,A))$, let $Z(V)$ denote
 the closed subscheme of $D$ defined by the vanishing of sections of
 $V$. Recall that $Z(V)$ has codimension at most $r$ in $D$. In fact,
 $Z(V)$ has codimension exactly $r$ in $D$ for a general
 $V\in G(r,H^0(D,A))$.

 The base locus of the linear system $\mbP V$ corresponding to $(A,V)$
 on $D$ is $Z(V)$. The ideal sheaf $\mcI$ of $Z(V)$ is the image of
 the morphism $V\otimes A^{\vee} \lr \mcO_D$ on $D$. This gives the
 surjective evaluation map
 $V\otimes\mcO_D\twoheadrightarrow A\otimes \mcI$ on $D$. Push-forward
 this morphism by the closed immersion $i$ to $X$, and consider the
 following composition:
\begin{equation*}\label{pushfwdseq}
  V\otimes\mcO_X\twoheadrightarrow V\otimes i_*\mcO_D\lr i_*(A\otimes \mcI)\lr 0\,.
\end{equation*}
Let $\mcF$ denote the kernel of the composition. Thus, we
get:
\begin{equation}\label{definingses}
 0\lr \mcF\lr V\otimes\mcO_X\lr i_*(A\otimes\mcI)\lr 0\,.
\end{equation}
\begin{prelim}\label{prelim}
 The kernel sheaf $\mcF$ associated to $(D,A,V)$ has the following initial properties:
 \begin{enumerate}
  \item[(a)] The sheaf $\mcF$ is reflexive of rank $r$.
  \item[(b)] For a general $V\in G(r,H^0(D,A))$, the sheaf $\mcF$ is locally free when $r\geq N$. 
  \item[(c)] The determinant of $\mcF$ is $\mcO_X(-D)\simeq\mcL^{\vee}$.
  \item[(d)] The sheaf $\mcF$ has no non-zero global sections, i.e. $H^0(X,\mcF)=0$.
 \end{enumerate}
\end{prelim}
 \begin{proof} 
   For part (a), we note that in the exact sequence
   \eqref{definingses}, $\mcF$ is the elementary transformation of a
   locally free sheaf by a torsion-free sheaf supported on the smooth
   divisor $D$. By \cite[Lemma 2.4]{Ab}, $\mcF$ is a reflexive sheaf
   of rank $r$. When $r\geq N$, a general $V\in G(r,H^0(D,A))$ has
   $\trm{codim}_D Z(V)=r\geq N$. Hence, $Z(V)$ is empty as $D$ is of
   dimension $N-1$. In this case, $A\otimes\mcI\,\simeq\, A$ and
   $\mcF$ is the usual elementary transformation, and is locally free.
   From \eqref{definingses},
   $\trm{det}\,\mcF\simeq\trm{det}\, i_*(A\otimes\mcI)^{\vee}$. Since
   $Z(V)$ is of codimension at least 2 in $X$, we have
   $\trm{det}\,i_*(A\otimes\mcI)\simeq \trm{det}\,i_*A\simeq
   \mcO_X(D).$ Part (d) can be proved by consider the long exact
   sequence of cohomology associated to \eqref{definingses}.
\end{proof}
This construction gives us reflexive sheaves of rank $r\geq 2$ on
smooth projective varieties of dimension $N\geq 2$. We call the dual
sheaves $\mcE_{D,A,V}=\mcF^{\vee}$, the \emph{Lazarsfeld-Mukai
  reflexive sheaves}. Dualizing the exact sequence \eqref{definingses}
defining $\mcF$, we get:
\begin{equation}\label{dualses}
  0\lr V^{\vee}\otimes\mcO_X\lr \mcE_{D,A,V}\lr\mathcal{E}xt^1(i_*(A\otimes\mcI),\mcO_X)\lr 0\,.
 \end{equation}
 As with Lazarsfeld-Mukai bundles, the Lazarsfeld-Mukai reflexive
 sheaves are naturally equipped with an $r$-dimensional space of
 global sections $V^{\vee}\subset H^0(X,\mcE_{D,A,V})$.
 \theoremstyle{definition}
 \newtheorem{Remarkonflatfamily}[Lefschetz]{Remark}
\begin{Remarkonflatfamily}\label{Remarkonflatfamily}
  Suppose $X$ is a smooth, irreducible projective variety of dimension
  $N\geq 2$ over $\mbb{C}$. Let $D$ be a smooth and irreducible
  divisor on $X$ and $A$ be an ample and globally generated line
  bundle on $D$. Let $r$ be such that $2\leq r \leq h^0(D,A)$.
  Consider the dense open subscheme $\mcU_r\subset G(r,H^0(D,A))$
  given by:
\begin{equation*}\label{correctcodim}
 \mcU_r=\{V\in G(r,H^0(D,A))\,|\,Z(V)\trm{ has codimension } r\trm{ in } D\}\,.
\end{equation*}
If $V\in\mcU_r$, it is well-known that the ideal sheaves $\mcI$ form a
flat family of sheaves parametrized by $\mcU_r$. Consequently, the
torsion-free sheaves $\{A\otimes \mcI\}_{V\in \mcU_r}$ and the dual
Lazarsfeld-Mukai reflexive sheaves $\{ \mcF \}_{V\in \mcU_r}$ form
flat families of sheaves parametrized by $\mcU_r$.
\end{Remarkonflatfamily}

\section{(Semi)stability of the sheaves $\mcF$}\label{studystab}
In this section we study the (semi)stability properties of the
reflexive sheaves $\mcF$.
\subsection{Varieties with cyclic Picard group}\label{picardgroupZ}
Suppose that $X$ is a smooth projective variety of dimension $N\geq 2$
such that $\trm{Pic}\,X\simeq \mathbb{Z}\cdot [H]$. Here $[H]$ is the
class of the ample generator the Picard group. Projective spaces,
smooth hypersurfaces in $\mbP^n$ for $n\geq 4$, general complete
intersections in higher dimensional projective spaces are some
examples of such varieties. We now prove Theorem
\ref{stabonprojintro}. When $X$ is a K3-surface with cyclic Picard
group and $\mcF$ is a vector bundle, the $\mu_H$-stability and the
simplicity of $\mcF$ is known, cf. \cite[Lemma 1.3]{RL}. We generalize
this to higher dimensional varieties when $\mcF$ is a reflexive sheaf.

\begin{proof}[Proof of Theorem \ref{stabonprojintro}.] 
  Denote $\mcE:=\mcE_{D,A,V}=\mcF^{\vee}$. Recall the exact sequence
  \eqref{dualses} defining the Lazarsfeld-Mukai reflexive sheaf. Since
  the cokernel sheaf $\mathcal{E}xt^1(i_*(A\otimes\mcI),\mcO_X)$ is
  supported only on $D$, there is a generically surjective morphism
  $\mcO_X^r\lr \mcE$.
  
  Assume the contrary, i.e. $\mcE$ (equivalently $\mcF$) is not
  $\mu_{H}$-stable. Then, there is a torsion-free quotient $Q$ of
  $\mcE$, i.e. $\mcE\twoheadrightarrow \mcQ$ of rank $s< r$, such that
  $\mu_{H}(\mcE)\geq\mu_{H}(\mcQ)$. From Proposition \ref{prelim},
  $\trm{det}\,(\mcE)\simeq\mcO_X(H)$,
  thus \[\frac{c_1(\mcQ)\cdot(H^{N-1})}{s}\leq\frac{H^N}{r}.\] Since
  $\trm{Pic}\,X\simeq\mathbb{Z}$, we get $c_1(\mcQ)\leq 0$.

  Let $\tildeQ=\mcQ^{\vee\vee}$, the double dual of the torsion-free
  sheaf $\mcQ$. The sheaf $\tildeQ$ on $X$ is reflexive with
  $c_1(\tildeQ)=c_1(\mcQ)$, and there is a natural inclusion
  $\mcQ\hra \tildeQ$ which is generically an isomorphism. We have
  morphisms $\mcO_X^r\lr \mcE\lr\mcQ\hra \tildeQ.$ As each morphism is
  generically surjective, we have a generically surjective morphism
  $\mcO_X^r\lr \tildeQ\,.$ The $s$-th wedge product of the above
  morphism gives a generically surjective morphism:
  \[\bigwedge^s\mcO_X^r\lr \bigwedge^s
    \tildeQ\simeq\trm{det}\,\tildeQ\,.\] By \cite[Proposition
  1.2.7]{HL}, we have
  $0=\mu_{H}(\bigwedge^s\mcO_X^r)\leq\mu_{H}(\trm{det}\,\tildeQ)=c_1(\trm{det}\,\tildeQ)\cdot
  (H^{N-1})\,,$ implying that $c_1(\tildeQ)\geq 0$. Hence,
  $c_1(\tildeQ)=0$,
  i.e. $\trm{det}\,Q\simeq\trm{det}\,\tildeQ\simeq\mcO_X$.

  Let $\xi$ be the generic point of $X$. The generically surjective
  morphism $\mcO_X^r\lr \tildeQ$, gives the surjective morphism of
  $\mcO_{\xi}$-vector spaces
  $(\mcO_X^r)_{\xi}\twoheadrightarrow \tildeQ_{\xi}$.  Suppose the
  morphism $\mcO_X^r\lr \tildeQ$ is given by global sections
  $t_1,t_2,\cdots,t_r$ of $\tildeQ$, then the stalks of these at $\xi$
  generate $\tildeQ_{\xi}$ as an $\mcO_{\xi}$-vector space. Thereby,
  there are $s$ among these, say $t_1,t_2,\cdots, t_s$ which form a
  basis of $\tildeQ_{\xi}$. These sections give a generically
  surjective morphism:
  \[f:\mcO_X^s\lr \tildeQ\,.\] The map $f$ is in fact an
  isomorphism. Indeed, if $K$ and $C$ denote the kernel and cokernel
  of $f$, we get:
  \[0\lr K\lr \mcO_X^s\xra{f} \tildeQ\lr C\lr 0\,.\] Since $f$ is
  generically surjective, $\trm{rank}\,C=0$. Then, $\trm{rank}\,K=0$,
  and thus $K=0$. Also, $\trm{det}\,C\simeq\mcO_X$. Therefore, $C$ is
  supported on a closed set of codimension $\geq 2$ on $X$. Hence,
  $\mcO_X^s$ and $\tildeQ$ are isomorphic on an open set whose
  complement has codimension $\geq 2$. By \cite[Proposition 1.6
  (iii)]{RH}, $\tildeQ\simeq\mcO_X^s$.

By \cite[Corollary 1.2]{RH}, $\mcQ^{\vee}$ is reflexive. This gives $\mcO_X^s\simeq \tildeQ^{\vee}\simeq\mcQ^{\vee\vee\vee}\simeq \mcQ^{\vee}$. From the surjection $\mcE\twoheadrightarrow\mcQ$, we get  $\mcQ^{\vee}\simeq\mcO_X^s\hra \mcF$. This contradicts $H^0(X,\mcF)=0$ (Proposition \ref{prelim} (d)). Hence, the kernel reflexive sheaf is $\mu_H$-stable. 
\end{proof}
\theoremstyle{plain}
\newtheorem{tftrivialdet}[stabonproj]{Remark}
\begin{tftrivialdet}
  From the proof of theorem, the following commutative diagram gives
  $Q\simeq \mcO_X^s$.
  \begin{displaymath}
  \xymatrix{\mcO_X^s\ar[r]^{\simeq} \ar@{->}[d] & \tildeQ\simeq \mcO_X^s\\
	    Q\ar@{^{(}->}[ur] &}
  \end{displaymath}
  The following result can be inferred from the above. Consider a
  smooth projective variety $X$. Let $Q$ be a torsion-free sheaf on
  $X$ with trivial determinant over an open set whose complement has
  codimension $\geq 2$. Suppose $Q$ admits a generically surjective
  morphism $\mcO_X^r\lr Q$, then the torsion-free sheaf $Q$ is itself
  trivial. We thank the referee for pointing this out.
\end{tftrivialdet}
\begin{stabonproj}\label{elementofpf}
  The proof of Theorem \ref{stabonprojintro} proves essentially the
  following statement: Suppose a smooth projective variety $X$ has
  $\emph{Pic}\,X=\mathbb{Z}\cdot[H]$, where $H$ is ample. Let $F$ be a
  reflexive sheaf on $X$ such that
\begin{enumerate}
\item[a.] the determinant $\emph{det}\,F=\mcO_X(-H)$,
\item[b.] there is a generically surjective morphism from a trivial
  bundle on $X$ to $F^{\vee}$,
\item[c.] the space of global sections of $F$, i.e. $H^0(X,F)=0$.
\end{enumerate}
Then the reflexive sheaf $F$ is $\mu_H$-stable.
\end{stabonproj}
\subsection{(Semi)stability in case of arbitrary smooth projective varieties}\label{irred}
Suppose that $X$ is an irreducible, smooth projective variety of
dimension $N\geq 2$ over $\mbb{C}$. Let $\mcL$ be an ample and
globally generated line bundle on $X$.
\begin{flatHtoD}\label{flatHtoD}
  For a general $D\in\emph{sm}|L|$, there is a finite, flat morphism
  $\phi:X\lr\mathbb{P}^N$ such that $D$ maps to the hyperplane
  $H=Z(x_0)$ in $\mathbb{P}^N$, where
  $x_0,x_1,\cdots,x_N\in H^0(\mbb{P}^N,\mcO(1))$ are the homogeneous
  coordinates.
\end{flatHtoD}
\begin{proof}
  Since $L$ is globally generated and $\trm{dim}\,X=N$, any general
  collection of $N+1$ sections in $H^0(X,L)$ generate $L$. If
  $D\in\trm{sm}|L|$ is general, then $D=Z(s_0)$ for some
  $s_0\in H^0(X,L)$, and we can find $s_1,s_2,\cdots, s_N\in H^0(X,L)$
  such that $\{s_0,s_1,s_2,\cdots,s_N\}$ generate $L$. These sections
  give a morphism $\phi:X\lr \mathbb{P}^N$ such that
  $\phi^*(\mcO(1))=L$ and $s_i=\phi^*x_i$. If $H$ is the hyperplane
  $Z(x_0)\subset\mbb{P}^N$, we get the following commutative diagram.
\begin{displaymath}
\xymatrix{D\,\ar@{^{(}->}[r]_{i}\ar[d]_{\phi|_D} & X\ar[d]^{\phi} \\
	  H\,\ar@{^{(}->}[r]_j & \mathbb{P}^N}
\end{displaymath}
Since $X$ is a smooth projective variety over $\mbb{C}$ and $\phi$ is
defined by the sections of an ample line bundle $L$, the morphism
$\phi$ is finite \cite[Corollary 1.2.15]{Laz}. This also implies that
$\phi$ is surjective. Thereby, $\phi$ is a flat morphism
\cite[Exercise III.9.3 (a)]{RH1}.
\end{proof}
We remark that, over positive characteristic we will obtain a finite
map from $X$ to the projective space by choosing a very ample line
bundle $L$.

Any ample, globally generated line bundle on the hyperplane $H$ is of
the form $\mcO_H(m)\simeq\mcO_{\mbb{P}^N}(m)|_H$ for some
$m>0$. Consider $V'\in G(r,H^0(H,\mcO_H(m))$ where
$2\leq r\leq h^0(H,\mcO_H(m))$, such that $\trm{codim}_H
Z(V')=r$. Then we have:
\begin{equation}\label{sesonH}
0\lr\mathcal{F}_{H,\mcO_H(m),V'}\lr V'\otimes {\mcO}_{\mbP^N}\lr {j}_*(\mcO_H(m)\otimes\mathcal{I}_{Z(V')})\lr 0\,.
\end{equation}
By Theorem \ref{stabonprojintro}, $\mathcal{F}_{H,\mcO_H(m),V'}$ is
$\mu_{\mcO(1)}$-stable. We now have the following lemma.
\begin{pbfromhpisstable}\label{pbfromhpisstable}
 Let $A=(L|_D)^{\otimes m}$ on $D$. Then
 \begin{enumerate}
  \item[(a)] the subspace $V=\phi|_D^*V'\subset H^0(D,A)$ and $\emph{codim}_{D} Z(V)=r$, 
  \item[(b)] the sheaf $\mcF\simeq \phi^*\mathcal{F}_{H,\mcO_H(m),V'}$ and is $\mu_L$-semistable.
 \end{enumerate}
\end{pbfromhpisstable}
\begin{proof} Note that,
  $A= (L|_D)^{\otimes m} \simeq \phi|_D^*\mcO_H(m)\,. $ This gives
  $V=\phi|_D^*V'\subset \phi|_D^*H^0(H,\mcO_H(m))\subset H^0(D,A)\,.$
  Since $V=\phi|_D^*V'$, the closed subscheme $Z(V)$ maps to $Z(V')$
  under $\phi$, and we get the following commutative diagram.
\begin{displaymath}
 \xymatrix{Z(V)\,\ar@{^{(}->}[r]^{i'} \ar[d]^{\phi|_{Z(V)}} & D \ar[d]^{\phi|_D}\ar@{^{(}->}[r]^{i} &  X\ar[d]^{\phi} \\
             Z(V') \ar@{^{(}->}[r]^{j'} & H \ar@{^{(}->}[r]^{j} & \mbb{P}^N}
\end{displaymath}
Since $\phi$ is a finite, surjective and flat morphism, so are
$\phi|_D$ and $\phi|_{Z(V)}$. This implies that $Z(V)$ and $Z(V')$
have the same dimension.  Thus,
$\trm{codim}_D(Z(V))=\trm{codim}_H(Z(V')) =r$. This proves (a).

Consider the pullback of the exact sequence \eqref{sesonH} by $\phi$:
\begin{equation}\label{pbfromhyperplane}
 0\lr \phi^*\mathcal{F}_{H,\mcO_H(m),V'}\lr V\otimes \mcO_X\lr \phi^*{j}_*(\mcO_H(m)\otimes\mathcal{I}_{Z(V')})\lr 0\,.
\end{equation}
By \cite[Proposition
III.9.3]{RH1},
$$\phi^*{j}_*(\mcO_H(m)\otimes\mathcal{I}_{Z(V')})\simeq
{i}_*(\phi|_D)^*(\mcO_H(m)\otimes\mathcal{I}_{Z(V')})\simeq
i_*(A\otimes \phi|_D^*\mathcal{I}_{Z(V')})\,.$$ Note that
$\phi|_D^*\,\mathcal{I}_{Z(V')}\simeq
\mathcal{I}_{Z(V)}\subset\mcO_D$.  Hence, the short exact sequence
\eqref{pbfromhyperplane} becomes
\[0\lr \phi^*\mathcal{F}_{H,\mcO_H(m),V'}\lr V\otimes \mcO_{X}\lr
  i_*(A\otimes\mathcal{I}_{Z(V)})\lr 0\,.\] Therefore,
$\phi^*\mathcal{F}_{H,\mcO_H(m),V'}\simeq\mcF$. Further, since
$\mathcal{F}_{H,\mcO_H(m),V'}$ is $\mu_{\mcO(1)}$-stable, and $\phi$
is a finite morphism, by \cite[Lemma 1.17]{Mar}, $\mcF$ is
$\mu_{L}$-semistable.
\end{proof}
We prove Theorem \ref{calabiyauintro}.
\begin{proof}[Proof of Theorem \ref{calabiyauintro}.] Let
  $D\in\trm{sm}|L|$ be a divisor which is general in the linear
  system. By Lemma \ref{flatHtoD}, there is a finite morphism
  $\phi:X\lr\mbb{P}^N$ such that $\phi(D)$ is the hyperplane
  $H=Z(x_0)$. Consider the line bundle $A=(L|_D)^{\otimes m}$ on $D$
  for any $m>0$.  By Remark \ref{Remarkonflatfamily}, there is a flat
  family of rank $r$ reflexive sheaves (where $2\leq r\leq h^0(D,A)$)
  parametrized by
  $V\in\mcU_r=\{V\in G(r,H^0(D,A))\,|\,\trm{codim}_D Z(V)=r\}$. Each
  $V\in\mcU_r$ corresponds to the reflexive sheaf
  $\mcF\simeq\mathcal{F}_{D,\mcL|_D^{\otimes m},V}$ .

  Let $V=\phi^*V'$, where $V'\in G(s,H^0(H,\mcO_H(m))$ is such that
  $\trm{codim}_H Z(V')=s$. Then by Lemma \ref{pbfromhpisstable},
  $V\in\mcU_s$ and $\mcF$ is $\mu_L$-semistable.  Note that $s$ can
  only vary in the range
  $2\leq s\leq h^0(H,\mcO_H(m))={N-1+m\choose m}$. Hence, for any $r$
  in the range $2\leq r\leq {N-1+m\choose m}$, we have a $V\in\mcU_r$
  with the corresponding $\mathcal{F}_{D,A,V}$ being
  $\mu_{L}$-semistable. By \cite[Proposition 2.3.1]{HL}, semistability
  is an open condition in flat families. Therefore, for a general
  $V\in G(r,H^0(D,A))$ where $2\leq r\leq {N-1+m\choose m}$, $\mcF$ is
  $\mu_L$-semistable.
\end{proof}
\subsection{(Semi)stability in case of projective space}\label{caseprojspace}
Suppose ${X=\mbP^N}$ for $N\geq 2$.
\subsubsection{} We first consider the case $L=\mcO(2)$ and $r=2$.
Let $D\in \trm{sm}|\mcO(2)|$ be a degree two hypersurface. Let $A$ be
an ample, globally generated line bundle on $D$. We now discuss the
$\mu_{\mcO(1)}$-stability of the rank 2 reflexive sheaf $\mcF$ where
$V\in G(2,H^0(D,A))$.

We recall the concept of normalization of a torsion-free sheaf
\cite[Chapter II, $\mathcal{x}\,$1.2]{OK}. A torsion-free sheaf $E$
of rank 2 over $X=\mathbb{P}^N$ has a uniquely determined integer
$k_E$ associated to it, namely,
\begin{equation}\label{eqn29}
k_E  =  -\frac{c_1(E)}{2} \trm{ if } c_1(E)\trm{  even, and } k_E=-\frac{c_1(E)+1}{2} \trm{ for } c_1(E)\trm{  odd}\,. 
\end{equation}
Note that, $c_1(E(k_E))\in \{0,-1\}$. We set $E_{\trm{norm}}:=E(k_E)$,
and call $E$ normalized if $E=E_{\trm{norm}}$.
\begin{reflexstab}\label{okolemma}
  \cite[Lemma II.1.2.5]{OK} A reflexive sheaf $E$ of rank 2 over
  $X=\mathbb{P}^N$ is stable if and only if $E_{\trm{norm}}$ has no
  sections : $H^0(X ,E_{\trm{norm}}) = 0.$ If $c_1 (E)$ is even, then
  $E$ is semistable if and only if
  $H^0(X , E_{\trm{norm}} (-1)) = 0. $
\end{reflexstab}
\begin{alterpf}\label{firstcase}
  Consider $L=\mcO(2)$ on $X=\mathbb{P}^N$ for $N\geq 2$. Let
  $D\in\emph{sm}|L|$, $A$ be an ample, globally generated line bundle
  on $D$ and $V\in G(2,H^0(D,A))$ be a 2-dimensional subspace. Then
  the associated rank two $\mcF$ is $\mu_{\mcO(1)}$-semistable.
\end{alterpf}
\begin{proof} If $\mcL=\mcO(2)$, then $c_1(\mcF)=-2$. Then,
  $k_{\mcF}=1.$ Thus, ${\mcF}_{\trm{norm}}=\mcF(1)$.  Now, the result
  follows from the fact that $H^0(X,\mcF)=0$ (Proposition \ref{prelim}
  (d)).
\end{proof}

\subsubsection{}
We consider the case $L=\mcO(2l)$ and $r=2$.

Start with $L=\mcO(d)$ for any $d>0$. Consider $D\in \trm{sm}|L|$
where $i:D\hra X$ denotes the inclusion. Suppose $A$ is an ample,
globally generated line bundle on $D$. For $r\geq 2$, let
$V\in G(r,H^0(D,A))$ be such that $\trm{codim}_D
Z(V)=r$.

We restrict the exact sequence \eqref{definingses} to the open set
$\tildeU=X\setminus Z(V)$ to get:
\begin{equation}\label{eq6}
 0\lr\mcF|_{\tildeU}\lr V\otimes\mcO_{\tildeU}\lr i_*(A\otimes\mcI)|_{\tildeU}\lr 0.
\end{equation}
Let $U=D\setminus Z(V)$, an open subset of $D$. By \cite[Proposition
II.6.5]{RH1}, $U$ is a divisor in $\tildeU$ and we have the following
commutative diagram.
\begin{displaymath}
\xymatrix{U\,\ar@{^{(}->}[r] \ar@{^{(}->}[d]_{i|_{U}} & D\, \ar@{^{(}->}[d]^i\\
\tildeU\,\ar@{^{(}->}[r] & X} 
\end{displaymath}
By \cite[Proposition III.9.3]{RH1},
$i_*(A\otimes\mcI)|_{\tildeU}\simeq
(i|_{U})_*(A|_{U}\otimes\mcI|_U)\simeq (i|_U)_*(A|_U)$. The short
exact sequence \eqref{eq6} then becomes:
\begin{equation*}
 0\lr\mcF|_{\tildeU}\lr V\otimes\mcO_{\tildeU}\lr (i|_{U})_*(A|_{U})\lr 0\,.
\end{equation*}
Thus by \cite[Lemma 1.1]{Ab}, $\mcF|_{\tildeU}$ is an elementary
transformation, and hence a locally free sheaf of rank $r$ on
$\tildeU$. Then, $(i|_{U})^*\mcF|_{\tildeU}=\mcF|_{U}$ on $U\subset D$
is also locally free of rank $r$. Now, $U$ is an open subset of $D$
whose complement $Z(V)$ is of codimension $r\geq 2$ in $D$. This
implies that the stalks of $\mcF|_D$ are torsion-free.
\begin{FonD}\label{FonD}
 The sheaf $\mcF|_D$ is torsion-free of rank $r$.
\end{FonD}
Assume now that $r=2$. Thus, $\mcF$ and $\mcF|_D$ have rank
2. Restrict the exact sequence \eqref{definingses} to $D$ to get:
\begin{equation*}\label{eq9}
 0\lr K\lr \mcF|_D\lr V\otimes\mcO_D\lr A\otimes\mcI\lr 0\trm{ on } D.
\end{equation*}
Here $K$ denotes the kernel, which is torsion-free by Remark
\ref{FonD}. Note that $K$ is a sheaf of rank 1 on $D$. Consider the
kernel $M$ of the surjective evaluation map
$V\otimes \mcO_D\twoheadrightarrow A\otimes\mcI$. The sheaf $M$ is
reflexive, cf. \cite[Proposition 1.1]{RH}. Since $M$ is of rank 1, by
\cite[Proposition 1.9]{RH}, $M$ is a line bundle. By comparing
determinants, $M\simeq A^{\vee}$, and we get:
\begin{equation*}
0\lr A^{\vee}\lr V\otimes\mcO_D\lr A\otimes\mcI\lr 0\,. 
\end{equation*}
From the exact sequences above, we get the following exact sequence on
$D$:
\begin{equation}\label{eq10}
 0\lr K\lr\mcF|_D\lr A^{\vee}\lr 0.
\end{equation}
\newtheorem{Rmktfree}[stabonproj]{Remark}
\begin{Rmktfree}\label{Remark}
  As $\mcF|_D$ is torsion-free,
  $\emph{det}\,\mcF|_D\simeq(\emph{det}\mcF)|_D\simeq(L|_D)^{\vee}$.
\end{Rmktfree}
Thus, from the exact sequence \eqref{eq10}, the determinant of the
rank 1 torsion-free sheaf $K$ is:
$$\trm{det}\,K\simeq A\otimes \trm{det}\,\mcF|_D\simeq A\otimes(L|_D)^{\vee}\,.$$
\newtheorem{sstabonD}[stabonproj]{Proposition}
\begin{sstabonD}\label{sstabonD}
  Consider $\mcL=\mcO(2l)$ for $l>0$. Let $D\in\emph{sm}|L|$,
  $A=\mcO(l)|_D$ and $V\in G(2,H^0(D,A))$ such that
  $\emph{codim}_D Z(V)=2$. Then the torsion-free sheaf $\mcF|_D$ on
  $D$ is $\mu_{\mcO(1)|_D}$-semistable. Thus $\mcF$ is
  $\mu_{\mcO(1)}$-semistable.
\end{sstabonD}
\begin{proof} Both $K$ and $A^{\vee}$ are rank one torsion-free
  sheaves on $D$ and hence are $\mu_{\mcO(1)|_D}$-stable. Further they
  have the same determinant. Indeed,
$$A^{\vee}\simeq \mcO(-l)|_D\trm{ and } \trm{det}\,K\simeq A\otimes(L|_D)^{\vee}\simeq \mcO(l)|_D\otimes \mcO(-2l)|_D\simeq\mcO(-l)|_D\,.$$

Hence $\mcF|_D$ is a torsion-free sheaf which is an extension of
$\mu_{\mcO(1)|_D}$-stable sheaves with the same slope. By \cite[Lemma
1.10 (3)]{Mar}, $\mcF|_D$ is $\mu_{\mcO(1)|_D}$-semistable. This
implies that $\mcF$ is $\mu_{\mcO(1)}$-semistable, cf. \cite[Chapter
11]{LP}.
\end{proof}
We collect our observations to prove Theorem \ref{finalprojintro}.
\begin{proof}[Proof of Theorem \ref{finalprojintro}.] Part (a) of the theorem follows from Theorem \ref{stabonprojintro}. Part (b) of the theorem follows from Corollary \ref{firstcase}. Proposition \ref{sstabonD} proves part (c) of the theorem. Finally, part (d) follows from Theorem \ref{calabiyauintro}. 
\end{proof}
\begin{exceptions}\label{exceptions}
  If $\mcL=\mcO(d)$, then for all $r$ such that $(r,d)=1$, a
  $\mu_{\mcO(1)}$-semistable $\mcF$ is in fact stable. Indeed, if
  $D\in |\mcO(d)|$, then
  $\emph{deg}\,\mcF=c_1(\mcF)\cdot (\mcO(1)^{N-1})=-d$ and
  $\emph{rank}\,\mcF=r$. Since the rank and degree are coprime, by
  \cite[Lemma 1.2.14]{HL}, semistability and stability coincide.
\end{exceptions}
Theorem \ref{finalprojintro} shows that there are families of
$\mu_{\mcO(1)}$-(semi)stable rank $r$ reflexive sheaves on the
projective space with any prescribed $c_1$.

The following lemma shows that if we weaken our condition on $A$ in
part (d) of Theorem \ref{finalprojintro}, then we may not get the
required (semi)stability.
\begin{notsemist}\label{notsemist}
  Consider $L=\mcO_X(d)$ ($d>0$) on $X=\mbb{P}^N$ for $N\geq 2$. Let
  $D\in |L|$ be a general smooth and irreducible hypersurface. Let
  $A=\mcO_X(l)|_D$ and $V\subset H^0(D,A)$ be a general
  $r$-dimensional subspace ($r\geq 2$) such that $r$ and $l$ satisfy
  the following condition:
  $$0<l<\frac{d(r-1)}{r}\,.$$
  Then $\mcF$ is not $\mu_{\mcO(1)}$-semistable.
 \end{notsemist}
 Note that if $l$ is an integer such that $0<l<\frac{d}{2}$, then for
 any $r\geq 2$, the above inequality is satisfied.
\begin{proof} Let $l$ and $r$ be integers that satisfy the given condition, i.e.
  \[0<l<\frac{d(r-1)}{r}\,.\] Consider the line bundle $\mcO_X(l)$ on
  the projective space $X$. Since $D\in |\mcO_X(d)|$ and
  $A=\mcO_X(l)|_D$, we have $H^0(X,\mcO_X(l))\simeq H^0(D,A)$. Choose
  a general $V\in G(r,H^0(X,\mcO_X(l)))\simeq G(r,H^0(D,A))$.

  Let $Z_X(V)$ (resp. $Z_D(V)$) denote the closed subscheme of $X$
  (resp. $D$), defined by the vanishing of sections of
  $V\subset H^0(X,\mcO_X(l))$ (resp. $V\subset H^0(D,A)$). In fact
  $Z_D(V)=Z_X(V)\cap D$. Note that for a general $D$ and $V$,
  $\trm{codim}_X Z_X(V)=\trm{codim}_D Z_D(V)=r$. We get the following
  exact sequence on $X$:
\begin{equation*}\label{ceg-1}
  0\lr M\lr V\otimes \mcO_X\lr \mcO_X(l)\otimes\mathcal{I}_{Z_X(V)}\lr 0\,.
\end{equation*}
Here $M$ is a torsion-free (in fact reflexive) sheaf of rank $r-1$ and
determinant $\mcO_X(-l)$. Thus $\mu_{\mcO_X(1)}(M)=\frac{-l}{r-1}$.

Also, corresponding to the triple $(D,A,V)$, we get the following
exact sequence on $X$ (where $i:D\hra X$ denotes the inclusion):
\begin{equation*}\label{ceg-2}
 0\lr \mcF\lr V\otimes\mcO_X\lr i_*(A\otimes\mcID)\lr 0\,.
\end{equation*}
Note that $\mcF$ is of rank $r$ and $\trm{det}\,\mcF=\mcO_X(-d)$. We
get $\mu_{\mcO_X(1)}(\mcF)=\frac{-d}{r}$.

Comparing the above exact sequences, there is an inclusion
$M\hra\mcF$. But
\[\mu_{\mcO_X(1)}(M)=\frac{-l}{r-1}>\frac{-d(r-1)}{r(r-1)}= \frac{-d}{r}=\mu_{\mcO_X(1)}(\mcF)\,.\]
Hence $\mcF$ is not $\mu_{\mcO(1)}$-semistable.
\end{proof}
\section{An application to Kernel bundles}\label{Applications}
In this section we see an application of the techniques used so
far. As before, we work over the field of complex numbers.

Consider the line bundle $\mcO(d)$ ($d>0$) on $\mbb{P}^n$, $n\geq
2$. We have the following exact sequence corresponding to the line
bundle $\mcO(d)$ where $M_{\mcO(d)}$ is the kernel vector bundle.
\begin{equation}\label{kernel1}
0\lr M_{\mcO(d)}\lr H^0(\mcO(d))\otimes\mcO_{\mbb{P}^n}\lr \mcO(d)\lr 0\,.
\end{equation}
Flenner \cite[Corollary 2.2]{HF} proved that the kernel bundles
$M_{\mcO(d)}$ are $\mu_{\mcO(1)}$-semistable.
\begin{hpstable}\label{hpstable}
  In fact, in case of the line bundle $\mcO(1)$ on $\mbb{P}^n$, the
  kernel bundle $M_{\mcO(1)}$ is stable. Indeed, note that
  $M_{\mcO(1)}$ is a locally free sheaf on $\mbb{P}^n$ with
  $\emph{det}\,M_{\mcO(1)}=\mcO(-1)$. Further, there is a surjective
  morphism from a trivial bundle on $\mbb{P}^n$ to
  $M_{\mcO(1)}^{\vee}$. For, dualizing \eqref{kernel1} when $d=1$, we
  get the exact sequence:
  \[0\lr\mcO(-1)\lr H^0(\mcO(1))^{\vee}\otimes \mcO_{\mbb{P}^n}\lr
    M_{\mcO(1)}^{\vee}\lr 0\,.\] Finally, since
  $H^0(X,M_{\mcO(1)})=0$, by Remark \ref{elementofpf}, $M_{\mcO(1)}$
  is $\mu_{\mcO(1)}$-stable.
\end{hpstable}
\begin{Kernelbundle}\label{kernel}
  Let $X$ be an irreducible, smooth projective variety of dimension
  $n$ over $\mbb{C}$. Consider an ample and globally generated line
  bundle $\mcL$ on $X$. For a general subspace $W\subset H^0(X,\mcL)$
  of dimension $n+1$, the kernel bundle $M_{\mcL,W}$ associated to
  $(\mcL,W)$ is $\mu_{\mcL}$-polystable.
\end{Kernelbundle}
\begin{proof} For a general $W\in G(n+1, H^0(X,L))$, the corresponding
  linear system $\mbP W$ is basepoint-free and gives a finite
  surjective morphism $\psi_W:X\lr \mbP^n\,.$ Note that
  $\psi_W^*\mcO(1)\simeq\mcL$ and that
  $\psi_W^*H^0(\mbP^n,\mcO(1))\simeq W$. We have the kernel bundle
  $M_{\mcO(1)}$ on $\mbb{P}^n$ defined by the exact sequence of the
  form \eqref{kernel1} for $d=1$. Pullback this sequence to $X$, we
  get:
\begin{equation*}
 0\lr \psi_W^*M_{\mcO(1)}\lr W\otimes \mcO_{X}\lr \mcL\lr 0\,.
\end{equation*}
Hence, $\psi_W^*M_{\mcO(1)}$ is the kernel bundle on $X$ associated to
$(\mcL,W)$, i.e. $\psi_W^*M_{\mcO(1)}\simeq M_{\mcL,W}$. Again, since
$M_{\mcO(1)}$ is $\mu_{\mcO(1)}$-stable, by \cite[Lemma 1.17]{Mar},
$M_{\mcL,W}$ is $\mu_L$-semistable. In fact, by a result of Kempf
\cite[Theorem 1]{GK}, $M_{\mcL,W}$ is $\mu_L$-polystable on $X$.
\end{proof}

\end{document}